\documentclass[11pt,amssymb]{amsart}
\usepackage{amssymb}
\usepackage[mathscr]{eucal}

\newcommand{\Q}{{\mathbb Q}}

\newcommand{\Ga}{\mathrm{Gal}}
\newtheorem{thm}{Theorem}[section]
\newtheorem{lemma}[thm]{Lemma}
\newtheorem{prop}[thm]{Proposition}
\newtheorem{cor}[thm]{Corollary}

\theoremstyle{definition}

.240pk scaled 1200 .240pk
\usepackage[all,cmtip]{xy}
\begin{document}

\title[Applications of the Fixed Point Theorem to algebraic groups]{Applications of the Fixed Point Theorem for group actions on buildings to algebraic groups over polynomial rings}

\author[P.~Abramenko]{Peter Abramenko}

\author[A.S.~Rapinchuk]{Andrei S. Rapinchuk}

\author[I.A.~Rapinchuk]{Igor A. Rapinchuk}

\address{Department of Mathematics, University of Virginia,
Charlottesville, VA 22904-4137, USA}

\email{pa8ex@virginia.edu}

\address{Department of Mathematics, University of Virginia,
Charlottesville, VA 22904-4137, USA}

\email{asr3x@virginia.edu}

\address{Department of Mathematics, Michigan State University, East Lansing, MI 48824, USA}

\email{rapinchu@msu.edu}

\begin{abstract}
We apply the Fixed Point Theorem for the actions of finite groups on Bruhat--Tits buildings and their products to establish two results concerning the groups of points of reductive algebraic groups over polynomial rings in one variable, assuming that the base field is of characteristic zero. First, we prove that for a reductive $k$-group $G$, every finite subgroup of $G(k[t])$ is conjugate to a subgroup of $G(k)$. This, in particular, implies that if $k$ is a finite extension of the $p$-adic field $\mathbb{Q}_p$, then the group $G(k[t])$ has finitely many conjugacy classes of finite subgroups, which is a well-known property for arithmetic groups. Second, we give a give a short proof of the theorem of Raghunathan--Ramanathan \cite{RagRam} about $G$-torsors over the affine line.
\end{abstract}

\maketitle

\hfill \parbox[t]{5cm}{\it In memoriam \newline Jacques Tits}

\vskip3mm

\section{Introduction}\label{S:I}

In 1872, in his very influential ``Erlanger Programm'', Felix Klein suggested to base the study of (certain) geometric spaces on the
analysis of their transformation groups, thus assigning groups a fundamental role in geometry. In the late 1950s and early 1960s,
Jacques Tits reversed this idea. He invented geometric structures that he called ``buildings'' in order to interpret (certain) groups of Lie
type as symmetry groups. Until then, some of these groups, particularly those of exceptional type, had been studied in a purely algebraic
manner. Through the use of buildings, important parts of modern group theory have become accessible to geometric methods and ideas, and
this has continued to be a crucial link between group theory and geometry, with many applications ever since.
The goal of this paper
is to present several results on the structure and Galois cohomology of the groups of points of absolutely almost simple algebraic groups over polynomial rings that rely in a critical way on the analysis of the actions of these groups on affine Bruhat--Tits buildings and the Fixed Point Theorem for actions of finite groups on these buildings and their products. We hope that this approach can be developed further to treat similar issues for the groups of points over the coordinate rings of more general affine curves.

First, one of the notable consequences of reduction theory for arithmetic groups is that every arithmetic subgroup has finitely many conjugacy classes of finite subgroups (cf. \cite{BoHC}, \cite[Ch. 4, Theorem 4.3]{PlRa}) --- we will refer to this property as (FC). Later, using a combination of various techniques, including the Fixed Point Theorem for group actions on CAT(0) spaces, it was shown in \cite[Theorem 1.4]{GrPl} that (FC) remains valid for \emph{all} finite extensions of arithmetic groups.
We will use the Fixed Point Theorem to establish (FC) for groups of the form $G(k[t])$, where $G$ is a reductive algebraic group over a nonarchimedean local field $k$ of characteristic zero, i.e.,~over a finite extension of some $p$-adic field $\mathbb{Q}_p$ (see Theorem  \ref{T:FC-1}).  This result can be viewed
as a contribution to the theory of algebraic groups over function fields of $p$-adic curves, where very significant progress has been achieved over the last decade on the cohomological Hasse principle and related issues --- see \cite{Par1}, \cite{Par2}, and references therein for the most recent installments of this work. Nevertheless, the question about (FC) for the groups of points over the coordinate rings of such curves has not been previously addressed. We plan to investigate this question for curves other than the affine line in future work. It should be pointed out that our finiteness result is derived from the following statement that is of independent interest.
\begin{thm}\label{T:I1}
Let $G$ be a reductive algebraic group over a field $k$ of characteristic 0. Then every finite subgroup of $G(k[t])$ is conjugate to a subgroup contained in $G(k)$.
\end{thm}
We note that if $p = \mathrm{char}\: k > 0$ and the derived group $[G , G]$ of $G$ is $k$-isotropic, then the group $G(k[t])$ contains finite abelian $p$-subgroups of unbounded orders (even when $k = \mathbb{F}_p$ is the field with $p$ elements). This implies that the above theorem fails in this situation and no finiteness theorem can be established unless we limit ourselves to finite subgroups of order prime to $p$.

Second, the result on (FC) for arithmetic groups  was used in \cite{BoSe} to establish the finiteness of the cohomology set $H^1(\mathfrak{g} , \Gamma)$ for {\it certain} actions of a finite group $\mathfrak{g}$ on an arithmetic group $\Gamma$, which was then applied to prove the properness of the global-to-local map in Galois cohomology in some situations. Subsequently, the finiteness of $H^1(\mathfrak{g} , \Gamma)$ was established in \cite[Theorem 1.5]{GrPl} for {\it all} actions of a finite group $\mathfrak{g}$ on an arithmetic group $\Gamma$ as a consequence of the result on (FC) for all finite extensions of arithmetic groups. We will utilize this approach, in conjunction with the techniques employed in the proof of Theorem \ref{T:I1}, to give a short proof of the theorem of Raghunathan--Ramanathan \cite{RagRam} (see Theorem \ref{T:RagRam}). The proof  quickly reduces to the following statement, which we then establish.
\begin{thm}\label{T:I2}
Let $G$ be a reductive algebraic group over a field $k$ of characteristic 0. Then for any finite Galois extension $\ell/k$, the natural map
$$
H^1(\mathrm{Gal}(\ell/k) , G(\ell)) \longrightarrow H^1(\mathrm{Gal}(\ell/k) , G(\ell[t]))
$$
is a bijection.
\end{thm}

It follows from this result that if $k$ is a local field of characteristic zero, then the set $H^1(\mathrm{Gal}(\bar{k}/k) , G(\bar{k}[t]))$ is finite. The result of \cite{CGP} on torsors over the punctured affine line implies the finiteness of the set  $H^1(\mathrm{Gal}(\bar{k}/k) , G(\bar{k}[t , t^{-1}])$
over such $k$. It would be interesting to see if the finiteness result remains valid over the coordinate rings of more general $p$-adic curves.

\section{The action of $G(k[t])$ on the relevant building}\label{S:building}

We begin this section with a summary of the results of B.~Margaux \cite{Marg} that generalize results established by C.~Soul\'e \cite{Soule} in the split case. It should be noted that for anisotropic groups, the relevant building degenerates into a point, but the results we need remain valid in this case due to Proposition \ref{P:anis} below.

\vskip1mm

\noindent {\bf 2.1. Results of B.~Margaux \cite{Marg}.} Let $G$ be an absolutely almost simple simply connected algebraic group over a field $k$. Fix a maximal $k$-split torus $S$ of $G$, and let $T$ be a maximal $k$-torus containing $S$. Take $\ell/k$ to be a finite Galois extension with Galois group $\mathscr{G} = \mathrm{Gal}(\ell/k)$ such that $T$ splits over $\ell$ (we do not assume that $\ell$ is necessarily the minimal splitting field of $T$).
Let $\Phi = \Phi(G , T)$ (resp., $\Phi_r = \Phi(G , S)$) denote the corresponding absolute (resp., relative) root system. We choose compatible orderings on $\Phi$ and $\Phi_r$, and denote by $\Delta$ and $\Delta_r$ the corresponding systems of simple roots.

Next, we set $K = k(\!(1/t)\!)$ and $L = \ell(\!(1/t)\!)$, and let $\mathscr{B}$ (resp., $\mathscr{B}_{\ell}$) denote the (affine) Bruhat--Tits building associated with the group $G_K := G \times_k K$ (resp., $G_L = G \times_k L$). Both $\mathscr{B}$ and $\mathscr{B}_{\ell}$ possess  natural simplicial complex structures, and there is an isometric embedding $j \colon \mathscr{B} \to \mathscr{B}_{\ell}$ that identifies $\mathscr{B}$ with the fixed set $(\mathscr{B}_{\ell})^{\mathscr{G}}$ under the natural action of $\mathscr{G}$. The hyperspecial subgroup $G(\ell[\![1/t]\!])$ of $G(L)$ fixes a unique vertex $p^0_{\ell} \in \mathscr{B}_{\ell}$. Since the subgroup $G(\ell[\![1/t]\!])$ is $\mathscr{G}$-invariant, the point $p^0_{\ell}$ is $\mathscr{G}$-fixed, hence $p^0_{\ell} = j(p^0)$ for $p^0 \in \mathscr{B}$. Let $\mathscr{A}$ (resp., $\mathscr{A}_{\ell}$) be the standard apartment of $\mathscr{B}$ (resp., $\mathscr{B}_{\ell}$) associated with $S$ (resp., $T_L$). Then
\begin{equation}\label{E:apart}
\mathscr{A} = p^0 + \mathrm{Hom}_{k\text{-}\mathrm{gr}}(\mathbb{G}_m , S) \otimes_{\mathbb{Z}} \mathbb{R} \ \ \text{and} \ \ \mathscr{A}_{\ell} = p^0_{\ell} + \mathrm{Hom}_{\ell\text{-}\mathrm{gr}}(\mathbb{G}_m , T) \otimes_{\mathbb{Z}} \mathbb{R},
\end{equation}
hence $j(\mathscr{A}) = (\mathscr{A}_{\ell})^{\mathscr{G}}$. Let
$$
\langle \cdot , \cdot \rangle \colon \mathrm{Hom}_{k\text{-}\mathrm{gr}}(S , \mathbb{G}_m) \times \mathrm{Hom}_{k\text{-}\mathrm{gr}}(\mathbb{G}_m , S) \to \mathbb{Z}
$$
be the canonical pairing. We then define the {\it sector} (quartier) by
$$
\mathscr{Q} = p^0 + D, \ \ \text{where} \ \ D:= \{ \lambda \in \mathrm{Hom}_{k\text{-}\mathrm{gr}}(\mathbb{G}_m , S) \otimes_{\mathbb{Z}} \mathbb{R} \ \vert \ \langle \alpha , \lambda \rangle \geq 0 \ \ \forall \alpha \in \Delta_r \}.
$$
The following theorem, which is Theorem 2.1 in \cite{Marg}, generalizes the result that was established by Soul\'e \cite{Soule} for split groups.
\begin{thm}\label{T:Marg}
The set $\mathscr{Q}$ is a simplicial fundamental domain for the action of $G(k[t])$ on $\mathscr{B}$. In other words, every simplex of $\mathscr{B}$ is equivalent under the action of $G(k[t])$ to a unique simplex of $\mathscr{Q}$.
\end{thm}

We note that in this paper, the uniqueness will be used only in the split case. It follows from the description that if $G$ is $k$-split, then $S = T$ and consequently ${\mathscr{A}}_{\ell} = \mathscr{A}$. Therefore, in the split case, every point of $\mathscr{B}_{\ell}$ is $G(\ell[t])$-equivalent to a point of $\mathscr{A}$.

Next, set $\Gamma = G(k[t])$. We will now recall the description of the stabilizers $\Gamma_x$ of points $x \in \mathscr{Q}$ obtained in \cite[\S 2]{Marg}.
Clearly, for $p^0 \in \mathscr{Q}$ we have $\Gamma_{p^0} = G(k)$. Let now $x \in \mathscr{Q} \setminus \{ p^0 \}$. Set
$$
I_x = \{ \alpha \in \Delta_r \ \vert \ \alpha(x) = 0 \}
$$
(we obviously have $I_{p^0} = \Delta$). As usual, for $I \subset \Delta_r$, we define
$$
S_I = \left( \bigcap_{\alpha \in I} \ker \alpha \right)^{\circ} \ \ \text{and} \ \  L_I = Z_G(S_I).
$$
Let $P_I$ be the standard parabolic subgroup of $G$ of type $I$ (cf. \cite[21.11]{Borel}). Then $P_I$ has $L_I$ as its (standard) Levi subgroup, i.e. $P_I$ is the semi-direct product $U_I \rtimes L_I$, where $U_I$ is the unipotent radical of $P_I$, which is a connected $k$-defined $k$-split unipotent subgroup of $P_I$. Recall that the root system $\Phi(L_I , S)$ coincides with $[I]$, the subsystem of $\Phi_r$ consisting of roots that are linear combinations of elements of $I$. Then $U_I$ is generated by the $U_{\alpha}$ for $\alpha \in \Phi_r^+ \setminus [I]$. (All these facts can be found in {\it loc. cit.}) The following result is a consequence of (\cite[Proposition 2.5(1)]{Marg}).
\begin{prop}\label{P:stab1}
We have
\begin{equation}\label{E:stab1}
\Gamma_x = (\Gamma_x \cap U_{I_x}(K)) \rtimes L_{I_x}(k),
\end{equation}
and $\Gamma_x \cap U_{I_x}(K)$ is the group of $k$-points of a $k$-split unipotent group.
\end{prop}

We will only explain how one can identify the group $\Theta_x := \Gamma_x \cap U_{I_x}(K)$  with the group $\mathscr{U}(k)$ of $k$-points of a $k$-defined split (unipotent) subgroup $\mathscr{U}$ of the restriction of scalars $\mathrm{R}_{A_n/k}(U_{I_x})$  for a sufficiently large positive integer $n$ where $A_n = k[t]/(t^n)$. It is one of the cornerstones of Bruhat--Tits theory that the stabilizers in $G(K)$ of points of the building are open bounded subgroups (cf. \cite[Axiom 4.1.2]{KaPr}) with respect to the topology associated with the valuation, in our case, with the valuation associated with $t^{-1}$. Thus, the subgroup $\Theta_x \subset U_{I_x}(k[t])$ is bounded, and therefore there exists an integer $n > 0$ such that $\Theta_x$ has trivial intersection with the congruence subgroup $U_{I_x}(k[t] , (t^n))$. Then $\Theta_x$ is identified with a subgroup of
$$
U_{I_x}(k[t]) / U_{I_x}(k[t] , (t^n)) \simeq U_{I_x}(A_n) \simeq \mathrm{R}_{A_n/k}(U_{I_x})(k),
$$
and in fact we have a similar identification when $n$ is replaced by any $n' \geq n$. Further analysis based on information about the structure of $\Theta_x$ that is developed in \cite[\S 2]{Marg} using the results of \cite{BT1} shows that the image of $\Theta_x$ coincides with the group of $k$-points of a $k$-defined subgroup $\mathscr{U}$ of $\mathrm{R}_{A_n/k}(U_{I_x})$ for a sufficiently large $n$. More precisely, let $\{b_1, \ldots , b_m \}$ be the set of reduced roots in $\Phi_r^+ \setminus [I_x]$. Then the product map $\prod_{j = 1}^m U_{b_j} \to U_{I_x}$ is an isomorphism of $k$-varieties \cite[Proposition 21.9]{Borel}, yielding an identification of $\prod_{j = 1}^m \mathrm{R}_{A_n/k}(U_{b_j})$ with $\mathrm{R}_{A_n/k}(U_{I_x})$. On the other hand, it follows from \cite[section 6.4.9]{BT1} that the product map $\prod_{j = 1}^n (\Gamma_x \cap U_{b_j}(K)) \to \Theta_x$ is a bijection. Thus, it is enough to prove that for a sufficiently large $n$, the subgroup $\Theta_{x,j} :=\Gamma_x \cap U_{b_j}(K)$ gets identified with the group of $k$-points of a $k$-subgroup $\mathscr{U}_j$ of $\mathrm{R}_{A_n/k}(U_{b_j})$. But $\Theta_{x,j}$ coincides with the intersection $\Gamma \cap \tilde{U}_{b_j , m}$ for some integer $m$, where $\tilde{U}_{b_j , m}$ is introduced on p.~395 of \cite{Marg}, and it is enough to show that the intersection $G(\ell[t]) \cap \tilde{U}_{b_j , m}$ gets identified with the group of $\ell$-points of a closed subgroup of $\mathrm{R}_{A_n/k}(U_{b_j})$ for a sufficiently large $n$. The decomposition defining $\tilde{U}_{b_j , m}$ in {\it loc. cit.} reduces this question to $G = \mathrm{SL}_2$, where it is well-known and easy to check (cf. \cite[Ch. II, \S 1]{Serre-Trees}).

Recall that for subsets $I , J \subset \Delta_r$, we clearly have $L_I \cap L_J = L_{I \cap J}$. So, we deduce from the proposition that given a subset $\Omega \subset \mathscr{Q}$, the pointwise stabilizer $\Gamma_{\Omega} := \bigcap_{x \in \Omega} \Gamma_x$ has the following description:
\begin{equation}\label{E:stab2}
\Gamma_{\Omega} = (\Gamma_{\Omega} \cap U_{\Omega}(K)) \rtimes L_{I_{\Omega}}(k), \ \ \text{where} \ \ U_{\Omega} = \bigcap_{x \in \Omega} U_{I_x} \ \ \text{and} \ \     I_{\Omega} = \bigcap_{x \in \Omega} I_x.
\end{equation}
Using the procedure we applied above to identify $\Theta_x$ with the group of $k$-points of a $k$-defined split unipotent subgroup,
we obtain the following.
\begin{cor}\label{C:stab1}
For any subset $\Omega \subset \mathscr{Q}$, the pointwise stabilizer $\Gamma_{\Omega}$ is of the form $\mathscr{U}(k) \rtimes L_{I_{\Omega}}(k)$, where $\mathscr{U}$ is a $k$-defined $k$-split unipotent group.
\end{cor}

\vskip.5mm

\noindent {\bf 2.2. A result concerning $G(k[t])$.} In order to extend some results from subsection 2.1 to arbitrary reductive groups, we need the following fact.
\begin{thm}\label{T:Stavr1}
Let $G$ be reductive algebraic $k$-group. Then
$$
G(k[t]) = G(k) \cdot G(k[t])^+,
$$
where $G(k[t])^+$ is the subgroup of $G(k[t])$ generated by the subgroups $U_P(k[t])$ for all minimal $k$-defined parabolic subgroups $P$ of $G$, with $U_P$ being the unipotent radical of $P$.
\end{thm}

This was established in \cite[Theorem 3.1]{Stav} under the additional assumption that every normal semi-simple $k$-subgroup of $[G , G]$ is $k$-isotropic. The argument in {\it loc. cit.} gives a reduction to the case where $G$ is semi-simple and simply connected, which  was considered in \cite{Marg} by generalizing the techniques introduced in \cite{Soule}. To handle the general case in the theorem, we will need the fact that if $G$ is $k$-anisotropic, then $G(k[t]) = G(k)$. In fact, we have the following more general statement.
\begin{prop}\label{P:anis}
Let $\tilde{C}$ be a smooth absolutely irreducible projective curve over a field $k$, let $P \in \tilde{C}(k)$ be a $k$-rational point, and $C = \tilde{C} \setminus \{ P \}$ be the corresponding affine curve. Then for any connected reductive algebraic $k$-group $G$ whose semi-simple part $H = [G , G]$ is $k$-anisotropic, we have $G(k[C]) = G(k)$.
\end{prop}
\begin{proof}
The claim is almost immediate if $G = T$ is a torus. Indeed, let $\ell$ be the splitting field of $T$. Then $\ell[C]^{\times} = \ell^{\times}$, and consequently
$$
T(k[C]) \subset T(\ell[C]) = T(\ell).
$$
So, $T(k[C]) \subset T(k[C] \cap \ell) = T(k)$.

We will reduce the proof to the case of a semi-simple $k$-anisotropic group. There is a central $k$-defined isogeny $\pi \colon G \to \bar{T} \times \bar{H}$ to the direct product of a torus and a semi-simple group. Clearly, $\pi$ yields a central isogeny $H \to \bar{H}$, so $\bar{H}$ is $k$-anisotropic. Assuming that the assertion of the proposition is valid for semi-simple $k$-anisotropic groups, we will have
$$
\pi(G(k[C])) \subset \bar{T}(k[C]) \times \bar{H}(k[C]) = \bar{T}(k) \times \bar{H}(k).
$$
Since $\ker \pi \subset G(\bar{k})$, we obtain $G(k[C]) \subset G(k[C] \cap \bar{k}) = G(k)$ as $k[C] \cap \bar{k} = k$ due to the fact that $C$ is absolutely irreducible (cf. \cite[Proposition 5.50]{GW}).

Now we will treat the main case where $G$ is a semi-simple $k$-anisotropic group. Let $K = k(C)$, and set $\mathcal{O}_P$  and $v = v_P$ to be the local ring of $P$ and the discrete valuation of $K$ associated with the point $P$, respectively. Then the completion $K_v$ can (and will) be identified with the field of formal power series $k(\!(t)\!)$, and $G$ remains $K_v$-anisotropic.

To see the latter, we recall that  the proper parabolic subgroups of $G$ are parametrized by a $k$-scheme $\mathscr{P}$ which is proper over $k$ (cf. \cite[\'{e}xpose XXVI, cor. 3.5 and 3.6]{SGA3}). The assumption that $G$ becomes isotropic over $K_v$ would mean that  $\mathscr{P}$ has a $K_v$-point. Then by the valuative criterion for properness (see, for
example, \cite[Theorem 15.9]{GW}) the scheme $\mathscr{P}$ would have a point over the valuation ring $\mathcal{O}_v$ of $K_v$. The reduction of this point would give us a point of $\mathscr{P}$ over the residue field, which in our case coincides with $k$. Thus, $\mathscr{P}(k) \neq \emptyset$, and hence $G$ has a proper $k$-defined parabolic. This means that $G$ is $k$-isotropic, which is not the case.

Next, fix a faithful $k$-defined representation $G \hookrightarrow \mathrm{GL}_n$. The fact that $G$ is $K_v$-anisotropic implies that $G(K_v)$ is a bounded subgroup of $\mathrm{GL}_n(K_v)$ (see \cite[Theorem 2.2.9]{KaPr}). We claim that in fact
\begin{equation}\label{E:IntegralX}
G(K_v) = G(\mathcal{O}_v),
\end{equation}
where $\mathcal{O}_v$ is the valuation ring of $K_v$. Indeed, suppose that $g = (g_{ij}) \in G(K_v) \setminus G(\mathcal{O}_v)$. Then for some indices $i_0 , j_0 \in \{1, \ldots , n\}$ we have $v(g_{i_0j_0}) < 0$. For each $m = 1, 2, \ldots$, we can consider the $k$-algebra homomorphism $\varphi_m \colon k(\!(t)\!) \to k(\!(t)\!)$ defined by sending $t$ to $t^m$. Since $G$ is defined over $k$, we have $(\varphi_m(g_{ij})) \in G(K_v)$ for all $m$. But the sequence $\varphi_m(g_{i_0j_0})$, $m = 1, 2, \ldots$, is unbounded, contradicting the boundedness of $G(K_v)$, hence proving (\ref{E:IntegralX}). Then
$$
G(k[C]) = G(k[C]) \cap G(\mathcal{O}_v) = G(k[C] \cap \mathcal{O}_P) = G(k[\tilde{C}]) = G(k),
$$
as required.
\end{proof}

\vskip2mm

\noindent {\bf Remarks.} 1. The statement of Proposition \ref{P:anis} in an earlier version of this paper included the assumption that $k$ has characteristic zero. This assumption was used in the proof to argue, using nilpotent elements in the Lie algebra, that $G$ remains $K_v$-anisotropic. Subsequently, Gopal Prasad showed us a justification of this fact over fields of any characteristic in the context of Bruhat--Tits theory as developed in \cite{KaPr}.
The current argument was proposed by the referee.

2. If a connected reductive $k$-group $G$ is $k$-anisotropic, then the same argument yields that $G(k[C]) = G(k)$ for any affine curve of the form $C = \tilde{C} \setminus \{P_1, \ldots , P_r\}$, where $P_1, \ldots , P_r \in \tilde{C}(k)$.

\vskip2mm

{\it Proof of Theorem \ref{T:Stavr1}.} We can write $G$ as an almost direct product $G = G_1 \cdot G_2$, where $G_1$ has the property that every semi-simple normal $k$-subgroup is $k$-isotropic and $G_2$ is $k$-anisotropic. Let $E = G_1 \cap G_2$, and set $\overline{G}_i = G_i/E$ for $i = 1, 2$ so that we have a $k$-isogeny
$$
\theta \colon G \to \overline{G}_1 \times \overline{G}_2,
$$
with $\ker \theta = E$. We have $\overline{G}_1(k[t]) = \overline{G}_1(k) \cdot \overline{G}_1(k[t])^+$ by \cite[Theorem 3.1]{Stav} and $\overline{G}_2(k[t]) = \overline{G}_2(k)$ by Proposition \ref{P:anis}. Since $\theta(G_1(k[t])^+) = \overline{G}_1(k[t])^+$, we see that
$$
G(k[t]) = (\theta^{-1}(\overline{G}_1(k) \times \overline{G}_2(k)) \cap G(k[t])) \cdot G(k[t])^+.
$$
However, $\theta^{-1}(\overline{G}_1(k) \times \overline{G}_2(k)) \subset G(\bar{k})$, and since $G(\bar{k}) \cap G(k[t]) = G(k)$, our claim follows. $\Box$

\vskip3mm

Now suppose that $G$ is a reductive $k$-group, and let $S$ be a maximal $k$-split torus, and $M = Z_G(S)$ be its centralizer. Furthermore, let $Z$ be the central torus (i.e. the connected center $Z(G)^{\circ}$) of $G$, let $H = [G , G]$ be the semi-simple part of $G$, and let $\pi_0 \colon \widetilde{H} \to H$ be a $k$-defined universal cover. Set $\widetilde{G} = \widetilde{H} \times Z$ and denote by $\pi \colon \widetilde{G} \to G$ the composition of the morphism $\widetilde{G} \to H \times Z$ induced by $\pi_0$ followed by the product map.

\begin{cor}\label{L:isog}
With the preceding notations, we have $G(k[t]) = M(k) \cdot \pi(\widetilde{G}(k[t]))$.
\end{cor}
\begin{proof}
According to Theorem \ref{T:Stavr1}, we have $G(k[t]) = G(k) \cdot G(k[t])^+$. On the other hand, $G(k) = M(k) \cdot G(k)^+$ by \cite[Proposition 6.11]{BorelTits}. So,
$$
G(k[t]) = M(k) \cdot G(k[t])^+.
$$
Since $G(k[t])^+ \subset \pi(\widetilde{G}(k[t]))$, we obtain our claim.
\end{proof}

\vskip1mm

\noindent {\bf 2.3. The case of $G$ simple, but not necessarily simply connected.} Let $G$ be an absolutely almost simple, but not necessarily simply connected, algebraic $k$-group, and let $\pi \colon \widetilde{G} \to G$ be a $k$-defined universal cover. It is well known that the Bruhat--Tits buildings associated with $\widetilde{G}$ and $G$ over $K$ are canonically isomorphic (cf. \cite[Remark 7.6.2, last paragraph]{KaPr}). We will denote this common building by $\mathscr{B}$. Then the canonical action of $G(K)$ on $\mathscr{B}$ composed with the group homomorphism $\pi(K) \colon \widetilde{G}(K) \to G(K)$ yields the canonical action of $\widetilde{G}(K)$. We will use $\widetilde{\ }$ to denote the objects associated with $\widetilde{G}$, dropping the tilde to denote the objects associated with $G$; in particular, we set $\widetilde{\Gamma} = \widetilde{G}(k[t])$ and $\Gamma = G(k[t])$. We fix a maximal $k$-split torus $\widetilde{S}$ of $\widetilde{G}$; then $S := \pi(\widetilde{S})$ is a maximal $k$-split torus of $G$. We let  $\mathscr{A} \subset \mathscr{B}$ denote the apartment constructed above for $\widetilde{S}$, and let $\mathscr{Q} \subset \mathscr{A}$ be the corresponding sector. We claim that Theorem \ref{T:Marg}, proved for the group $\widetilde{\Gamma}$, remains valid for $\Gamma$, i.e. $\mathscr{Q}$ is a simplicial fundamental domain for the group $\Gamma$. To see this, we observe that according to Corollary \ref{L:isog} we have $\Gamma = M(k) \cdot \pi(\widetilde{\Gamma})$. Then the mere inclusion $\pi(\widetilde{\Gamma}) \subset \Gamma$ yields the fact that $\mathscr{B} = \Gamma \cdot \mathscr{Q}$. To continue the argument, we recall that $M(k)$ acts trivially on the entire apartment $\mathscr{A}$, which follows, for example from \cite[Proposition 9.3.16]{KaPr}.
So, if two simplices $\mathscr{S}_1 , \mathscr{S}_2 \subset \mathscr{Q}$ are related by $\mathscr{S}_2 = \gamma(\mathscr{S}_1)$ with $\gamma \in \Gamma$, then writing $\gamma = m \cdot \delta$ with $m \in M(k)$ and $\delta \in \pi(\widetilde{G}(k[t]))$, we will have $\delta(\mathscr{S}_1) = \mathscr{S}_2$. Then $\mathscr{S}_1 = \mathscr{S}_2$ by Theorem \ref{T:Marg}, completing the argument.

\vskip.5mm

\vskip1mm

We will now show that the description of pointwise stabilizers of subsets of $\mathscr{Q}$ obtained in Corollary \ref{C:stab1} for simply connected groups remains valid in the general case. Let $\Omega \subset \mathscr{Q}$ be an arbitrary subset. We have seen in \S2.1 that $\widetilde{\Gamma}_{\Omega} = (\Gamma_{\Omega} \cap \widetilde{U}(K)) \rtimes \widetilde{L}_{I_{\Omega}}(k)$ in the notations introduced there. On the other hand, by Corollary \ref{L:isog}, we have $\Gamma = M(k) \cdot \pi(\widetilde{\Gamma})$. As we already mentioned in the previous paragraph,  $M(k)$ acts on $\mathscr{A}$ trivially, so we conclude that $\Gamma_{\Omega} = M(k) \cdot \pi(\widetilde{\Gamma}_{\Omega})$. The isogeny $\pi$ induces a $k$-isomorphism $\widetilde{U}_{\Omega} \to U_{\Omega}$, and hence group isomorphisms
$$
\widetilde{\Gamma} \cap \widetilde{U}_{\Omega}(K) \to \Gamma \cap U_{\Omega}(K) \ \ \text{and} \ \ \widetilde{\Gamma}_{\Omega} \cap \widetilde{U}_{\Omega}(K) \to \Gamma_{\Omega} \cap U_{\Omega}(K).
$$
Thus, $\pi(\widetilde{\Gamma}_{\Omega}) = (\Gamma_{\Omega} \cap U_{\Omega}(K)) \rtimes \pi(\widetilde{L}_{\Omega}(k))$. Being a subgroup of $\Gamma_{\Omega}$, the group $M(k)$ normalizes the intersection $\Gamma_{\Omega} \cap U_{\Omega}(K)$. Hence
\begin{equation}\label{E:stab10}
\Gamma_{\Omega} = M(k) \cdot \pi(\widetilde{\Gamma}_{\Omega}) = (\Gamma_{\Omega} \cap U_{\Omega}(K)) \rtimes L_{I_{\Omega}}(k)
\end{equation}
as $L_{I_{\Omega}}(k) = M(k) \cdot \pi(\widetilde{L}_{I_{\Omega}}(k))$ (cf. \cite[Proposition 6.11]{BorelTits}). Furthermore, in view of Corollary \ref{C:stab1}, the isomorphism $\widetilde{\Gamma}_{\Omega} \cap \widetilde{U}_{\Omega} \to \Gamma_{\Omega} \cap U_{\Omega}$ enables us to identify the latter with the group $\mathscr{U}(k)$ of $k$-points of a $k$-defined $k$-split unipotent subgroup $\mathscr{U}$.


\vskip1mm

To close this subsection, we now consider one special situation that will come up in the proof of the Raghunathan--Ramanathan theorem in \S \ref{S:RR}. Let $G$ be an absolutely almost simple  $k$-group that is {\it quasi-split} over $k$. As usual, we denote by $S$ a maximal $k$-split torus of $G$. Then the centralizer $M = Z_G(S)$ is a maximal $k$-torus $T$ of $G$ (cf., e.g., \cite[Proposition 16.2.2]{Springer}). Let $\ell/k$ be a finite Galois extension with Galois group $\mathscr{G} = \mathrm{Gal}(\ell/k)$ that contains the splitting field of $T$. Let $\Delta \subset \Phi(G , T)$ be a $\mathscr{G}$-invariant system of simple roots that corresponds to a $k$-defined Borel subgroup of $G$ containing $T$. Let $\mathscr{Q}_{\ell}$ be the sector in the apartment $\mathscr{A}_{\ell}$ defined using this $\Delta$; clearly, $\mathscr{Q}_{\ell}$ is $\mathscr{G}$-invariant.
\begin{prop}\label{P:stab2}
With assumptions and notations as above, let $\Gamma_{\ell} = G(\ell[t])$. Then for any $\mathscr{G}$-invariant subset $\Omega \subset \mathscr{Q}_{\ell}$, the pointwise stabilizer $(\Gamma_{\ell})_{\Omega}$ is of the form $((\Gamma_{\ell})_{\Omega} \cap U_{\Omega}(L)) \rtimes L_{I_{\Omega}}(\ell)$, where $L_{I_{\Omega}}$ is a $k$-defined reductive subgroup and $U_{\Omega}$ is a $k$-defined $k$-split unipotent subgroup, of $G$. Consequently, $(\Gamma_{\ell})_{\Omega}$ can be identified with the group of $\ell$-points of a $k$-defined semi-direct product $\mathscr{U} \rtimes L_{I_{\Omega}}$ where $\mathscr{U}$ is a $k$-defined $k$-split unipotent group.
\end{prop}

The description of $(\Gamma_{\ell})_{\Omega}$ follows from our earlier discussion based on results of Margaux \cite{Marg}, and the stated properties of $L_{I_{\Omega}}$ and $U_{\Omega}$ are clear. Elaborating on the procedure described after Proposition \ref{P:stab1}, one can identify the intersection $(\Gamma_{\ell})_{\Omega} \cap U_{\Omega}(L)$ with the group of $\ell$-points of a $k$-defined $k$-split subgroup $\mathscr{U}$ of $\mathrm{R}_{A_n/k}(U_{\Omega})$ for a sufficiently large $n$, where $A_n = k[t]/(t^n)$. The group $L_{I_{\Omega}}$ naturally acts on $U_{\Omega}$ by conjugation, allowing us to form a $k$-defined semi-direct product $U_{\Omega} \rtimes L_{I_{\Omega}}$. Applying restriction of scalars, we obtain
$$
\mathrm{R}_{A_n/k}(U_{\Omega} \rtimes L_{I_{\Omega}}) = \mathrm{R}_{A_n/k}(U_{\Omega}) \rtimes \mathrm{R}_{A_n/k}(L_{I_{\Omega}}).
$$
The structure homomorphism $k \to A_n$ yields a $k$-morphism $L_{I_{\Omega}} \to \mathrm{R}_{A_n/k}(L_{I_{\Omega}})$, leading to the semi-direct product $\mathrm{R}_{A_n/k}(U_{\Omega}) \rtimes L_{I_{\Omega}}$. Then the intersection $(\Gamma_{\ell})_{\Omega} \cap U_{\Omega}(L)$ is identified with the group of $\ell$-points of the subgroup $\mathscr{U} \rtimes L_{I_{\Omega}}$ of this semi-direct product.

\vskip1mm

\noindent {\bf 2.4. Arbitrary reductive groups.} Let $G$ be a reductive $k$-group, $Z = Z(G)^{\circ}$ be the central torus, and $H = [G , G]$ be the semi-simple part. Let $H_1, \ldots , H_r$ be the $k$-simple components of $H$. For each $i = 1, \ldots , r$, we let $\widetilde{H}_i$ (resp., $\overline{H}_i$) denote the corresponding simply connected (resp., adjoint) group. Also, let $F = Z \cap H$ and $\overline{Z} = Z/F$. Set
$$
\widetilde{G} = Z \times \widetilde{H}_1 \times \cdots \times \widetilde{H}_r \ \ \text{and} \ \ \overline{G} = \overline{Z} \times \overline{H}_1 \times \cdots \times \overline{H}_r.
$$
We then have the evident $k$-isogenies $\pi_1 \colon \widetilde{G} \to G$ and $\pi_2 \colon G \to \overline{G}$. Next, for each $i = 1, \ldots , r$, we can write
$$
\widetilde{H}_i = \mathrm{R}_{\ell_i/k}(\widetilde{\mathscr{H}}_i) \ \ \text{and} \ \ \overline{H}_i = \mathrm{R}_{\ell_i/k}(\overline{\mathscr{H}}_i)
$$
for absolutely almost simple simply connected and adjoint groups $\widetilde{\mathscr{H}}_i$ and $\overline{\mathscr{H}}_i$ and some finite separable extensions $\ell_i/k$. Set $L_i := \ell_i(\!(1/t)\!)$ and let $\mathscr{B}_i$ be the Bruhat--Tits building associated with $\widetilde{\mathscr{H}}_i$ and $\overline{\mathscr{H}}_i$ over $L_i$. The natural identifications
$$
\overline{H}_i(K) \simeq \overline{\mathscr{H}}_i(L_i)
$$
enable us to define an action of $\overline{G}(K)$ on the product of buildings $\mathscr{B} = \mathscr{B}_1 \times \cdots \mathscr{B}_r$, and hence an action of $G(K)$ that factors through $\pi_2$. (We observe that then the natural action of $\widetilde{G}(K)$ on $\mathscr{B}$ factors through $\pi_1$.) Furthermore, for each $i = 1, \ldots , r$, one can fix a maximal $\ell_i$-split torus $\mathscr{S}_i$ of $\mathscr{H}_i$ and a system of simple roots $\Delta_i \subset \Phi(\widetilde{\mathscr{H}}_i , \mathscr{S}_i)$, and then consider the corresponding apartments $\mathscr{A}_i$ of $\mathscr{B}_i$ and the sectors $\mathscr{Q}_i \subset \mathscr{A}_i$. Set
$$
\mathscr{A} = \mathscr{A}_1 \times \cdots \times \mathscr{A}_r \ \ \text{and} \ \ \mathscr{Q} = \mathscr{Q}_1 \times \cdots \times \mathscr{Q}_r.
$$
\begin{prop}\label{P:genred}
{\rm (1)} Every point of $\mathscr{B}$ is equivalent under the group $\Gamma = G(k[t])$ to a point of $\mathscr{Q}$.

\vskip1mm

\noindent {\rm (2)} For any $x \in \mathscr{Q}$, the stabilizer $\Gamma_x$ is of the form $U(k) \rtimes L(k)$ for some reductive $k$-subgroup $L \subset G$ and some unipotent $k$-group $U$.
\end{prop}
\begin{proof}
Part (1) follows from the fact that every point of $\mathscr{B}$ is equivalent to a point of $\mathscr{Q}$ under the action of $\widetilde{\mathscr{H}}_1(\ell_1[t]) \times \cdots \times \widetilde{\mathscr{H}}_r(\ell_r[t])$, which is a consequence of Theorem \ref{T:Marg}.
To prove part (2), one observes that the group $\widetilde{\Gamma} = \widetilde{G}(k[t])$ admits the following decomposition
$$
\widetilde{\Gamma} = Z(k) \times \tilde{\Gamma}_1 \times \cdots \times \tilde{\Gamma}_r \ \ \text{where} \ \ \tilde{\Gamma}_i = \mathscr{H}_i(\ell_i[t]).
$$
If $x = (x_1, \ldots , x_r)$ then clearly $\widetilde{\Gamma}_x = Z(k) \times (\widetilde{\Gamma}_1)_{x_1} \times \cdots \times (\widetilde{\Gamma}_r)_{x_r}$. Proposition \ref{P:stab1} provides a description of each stabilizer $(\widetilde{\Gamma}_i)_{x_i}$, which yields a description of $\widetilde{\Gamma}_x$. To obtain the required description of $\Gamma_x$, one observes that $\Gamma_x = M(k) \cdot \pi_1(\widetilde{\Gamma}_x)$ where $M$ is the centralizer of the maximal $k$-split torus of $G$ associated with $\mathscr{Q}$ (see Corollary \ref{L:isog}) and argues exactly as in \S2.3.
\end{proof}

We note that the result remains valid when some of the $k$-simple components are $k$-anisotropic.

\vskip1mm

\noindent {\bf 2.5. Type-preserving automorphisms.} We refer the reader to \cite{AbBr} for general background on buildings. We recall that given a building $\mathscr{B}$, every apartment $\mathscr{A}$ of $\mathscr{B}$ is isomorphic to a standard Coxeter complex $\Sigma(\mathscr{W} , \mathscr{S})$, where the Coxeter system $(\mathscr{W} , \mathscr{S})$ is uniquely determined by $\mathscr{B}$ and called the {\it type} of $\mathscr{B}$. There exists a {\it type function} $\tau$ defined on the set of vertices of $\mathscr{B}$ with values in $\mathscr{S}$ such that for every chamber $\mathscr{C}$ of $\mathscr{B}$, the restriction of $\tau$ to the set of vertices of $\mathscr{C}$ is a bijection onto $\mathscr{S}$ (see \cite[Proposition 4.6]{AbBr}; we note that ``to be colorable'' precisely means ``to admit a type function,'' cf. Definition A.10 in {\it loc. cit.}). Since $\mathscr{B}$ is a chamber complex, $\tau$ is uniquely determined by its restriction to the set of vertices of a single chamber. To every simplicial automorphism $\phi$ of $\mathscr{B}$ one can associate a permutation $\pi = \pi(\phi)$ of $\mathscr{S}$ such that
$$
\tau(\phi(v)) = \pi(\tau(v)) \ \ \text{for all vertices} \ \ v \ \ \text{of} \ \ \mathscr{B}
$$
(this follows, for example, from Proposition A.14 of Section A.1.3 in \cite{AbBr}). Then $\phi$ is called {\it type-preserving} if $\pi(\phi) = \mathrm{id}_{\mathscr{S}}$. We note that in order to establish that $\phi$ is type-preserving, it is enough to show that $\tau(\phi(v)) = \tau(v)$ for all vertices $v$ of a single chamber $\mathscr{C}$ of $\mathscr{B}$. Here is one simple but important fact concerning type-preserving automorphisms.
\begin{lemma}\label{L:TP}
Let $\phi$ be a type-preserving automorphism of a building $\mathscr{B}$,  let $X = \vert \mathscr{B} \vert$ be the geometric realization
of $\mathscr{B}$, and let $f \colon X \to X$ be the homeomorphism induced by $\phi$.

\vskip1mm

\noindent {\rm (1)} \parbox[t]{16cm}{If $\Sigma$ is a simplex of $\mathscr{B}$ such that $\phi(\Sigma) = \Sigma$, then $\phi(v) = v$ for every vertex $v$ of $\Sigma$, and hence $f$ acts trivially on $\vert \Sigma \vert$.}

\vskip1mm

\noindent {\rm (2)} \parbox[t]{16cm}{If $f$ fixes a point of $X$, then $\phi$ fixes a vertex of $\mathscr{B}$.}

\end{lemma}
\begin{proof}
(1): Let $\mathscr{C}$ be a chamber containing $\Sigma$. Since $\phi(\Sigma) = \Sigma$, for every vertex $v$ of $\Sigma$, the image $w = \phi(v)$ is also a vertex of $\Sigma$, hence of $\mathscr{C}$. But $\tau(w) = \tau(\phi(v)) = \tau(v)$ as $\phi$ is type-preserving. Since the restriction of $\tau$ to the set of vertices of $\mathscr{C}$ is injective, we conclude that $w = v$, i.e. $\phi(v) = v$.

\vskip1mm

(2): Assume that $f(x) = x$ for a point $x \in X$. Then there is a unique simplex $\Sigma_x$ in $\mathscr{B}$ such that $\vert \Sigma_x \vert$ is the carrier $F_x$ of $x$, i.e. the closed cell/simplex in $X$ that contains $x$ in its interior. We have $F_x = F_{f(x)} = f(F_x)$, which implies that $\phi(\Sigma_x) = \Sigma_x$. By part (1), $\phi$ (hence also $f$) fixes every vertex of $\Sigma_x$.
\end{proof}



We now return to the notations introduced in subsection 2.3. It follows from \cite[Proposition 5.2.10(i), p.~165]{BT2} and standard results about $BN$-pairs and their associated buildings (cf. \cite[Theorem, p. 80]{Garrett}) that  the group $\widetilde{G}(K)$ acts on $\mathscr{B}$ by type-preserving transformations, but this may not be true for the action of $G(K)$. Nevertheless, we have the following.
\begin{prop}\label{P:TP}
Let $G$ be an absolutely almost simple algebraic $k$-group. Then $G(k[t])$ acts on $\mathscr{B}$ by type-preserving transformations.
\end{prop}
\begin{proof}
As we mentioned above, the action of $\tilde{G}(k[t])$ on $\mathscr{B}$ is type-preserving  and factors through the isogeny $\pi \colon \widetilde{G} \to G$. On the other hand, by Corollary \ref{L:isog} we have $G(k[t]) = M(k) \cdot \pi(\widetilde{G}(k[t]))$ where $M$ is the centralizers of a fixed maximal $k$-split torus $S$ of $G$. So, it is enough to show that $M(k)$ acts on $\mathscr{B}$ by type-preserving transformations. As we mentioned at the start of \S2.3,
$M(k)$ acts trivially on the apartment $\mathscr{A} \subset \mathscr{B}$  corresponding to $S$. But for any vertex $v$ of $\mathscr{B}$, there exists a $g \in G(k[t])$ (and even in $\pi(\tilde{G}(k[t]))$) such that $gv \in \mathscr{A}$ (see \S2.3). Then for any $m \in M(k)$ we have $m(gv) = gv$, so
$$
mv = (mg^{-1}m^{-1}g)v.
$$
Let $C = \ker \pi.$ Since $C$ is central, it follows that the coboundary map $G(k[t]) \to H^1(\mathcal{G}_k, C(\bar{k}[t]))$ is a group homomorphism (where $\mathcal{G}_k = \Ga(\bar{k}/k)$) and we have an exact sequence of groups
$$
\tilde{G}(k[t]) \stackrel{\pi_{k[t]}}{\longrightarrow} G(k[t]) \to H^1(\mathcal{G}_k, C(\bar{k}[t]))
$$
(see, e.g., \cite[Ch. 1, \S1.3.2]{PlRa}).
From this, we see that $\pi(\widetilde{G}(k[t]))$ is a normal subgroup of $G(k[t])$, and the quotient $G(k[t])/\pi(\widetilde{G}(k[t]))$ is abelian, so
$$
\pi(\widetilde{G}(k[t])) \supset [G(k[t]) , G(k[t])].
$$
In particular, the element $mg^{-1}m^{-1}g$ lies in $\pi_{k[t]}(\tilde{G}(k[t]))$, and since the latter acts by type-preserving transformations, we conclude that $m$ also acts by type-preserving transformations, as required.
\end{proof}

Although we differentiated between a building and its geometric realization in the above discussion, we do not make this distinction elsewhere in the paper.

\vskip1mm

\noindent {\bf 2.6. The Fixed Point Theorem.} The proofs of the main results of this paper critically depend on the following
Fixed Point Theorem. We recall that a Bruhat--Tits building is equipped with a canonical metric and is a complete CAT(0) space for that metric
(cf. \cite[11.2]{AbBr}). So, the Fixed Point Theorem for isometric actions of groups on CAT(0) spaces with a bounded orbit \cite[Theorem 11.23]{AbBr} has the following consequence.

\begin{thm}\label{T:FPT}
Let $\mathscr{B}_1, \ldots , \mathscr{B}_r$ be Bruhat--Tits buildings, and suppose that a finite group $\Gamma$ acts on each building $\mathscr{B}_i$ by isometries. Then there exists a fixed point for the diagonal action of $\Gamma$ on the product $\mathscr{B} := \mathscr{B}_1 \times \cdots \times \mathscr{B}_r$.
\end{thm}

\begin{proof}
It follows from the Fixed Point Theorem for CAT(0) spaces that for each $i = 1, \ldots , r$, there exists a $\Gamma$-fixed point $b_i^0 \in \mathscr{B}_i$. Then $b^0 = (b_1^0, \ldots , b_r^0)$ is a fixed point for the diagonal action of $\Gamma$ on $\mathscr{B}$.
\end{proof}

(Even though it is not needed in the proof of the theorem, one should keep in mind that given (complete) CAT(0) metric spaces $(X_1 , d_1), \ldots , (X_r , d_r)$, their product $X := X_1 \times \cdots \times X_r$ is a (complete) CAT(0) space for the metric $d$ defined by $d^2 = d_1^2 + \cdots + d_r^2$.)


\section{Some auxiliary results}

In this section, we establish several auxiliary statements needed for the proof of the main results. The reader is referred to \cite[Ch.~I, \S 5]{Serre-GC} or \cite[1.3.2]{PlRa} for the basics of nonabelian cohomology. We begin with the following well-known lemma.

\begin{lemma}\label{L:2}
Let $G = N \rtimes H$ be a semi-direct product of abstract groups, and let $\Psi \subset H$ be a subgroup. Then a map
$f_{\xi} \colon \Psi \to G$ of the form $\delta \mapsto (\xi(\delta) , \delta)$ is a group homomorphism if and only if
the map $\xi \colon \Psi \to N$ is a 1-cocycle. Furthermore, two such homomorphisms $f_{\xi_1}$ and $f_{\xi_2}$ are conjugate by
an element $n \in N$ if and only if the corresponding cocycles $\xi_1$ and $\xi_2$ are equivalent.
\end{lemma}
\begin{proof}
We find by direct computation that for $\delta_1 , \delta_2 \in \Psi$, we have
$$
f_{\xi}(\delta_1)f_{\xi}(\delta_2) = (\xi(\delta_1) , \delta_1) (\xi(\delta_2) , \delta_2) = (\xi(\delta_1) \cdot (\delta_1 \xi(\delta_2)) , \delta_1 \delta_2), $$
and our first assertion follows. Furthermore,
$$
(n , 1)^{-1} f_{\xi}(\delta) (n , 1) = ( n^{-1} \cdot \xi(\delta) \cdot (\delta n) , \delta),
$$
yielding our second assertion.
\end{proof}

\vskip1mm

\begin{lemma}\label{L:3}
Let $G = N \rtimes H$ be a semi-direct product and $\pi \colon G \to H$ be the canonical projection. If $\Gamma \subset G$ is a subgroup such that
$\Gamma \cap N = \{ e \}$ and $H^1(\pi(\Gamma) , N) = 1$, then $\Gamma$ is conjugate in $G$ to a subgroup of $H$.
\end{lemma}
\begin{proof}
Set $\Psi = \pi(\Gamma)$. By our assumption, $\pi$ induces an isomorphism $\Gamma \to \Psi$; let $f \colon \Psi \to G$ be the
inverse of this isomorphism. It follows from Lemma \ref{L:2} that $f$ is of the form $f_{\xi}(\delta) = (\xi(\delta) , \delta)$ for some
1-cocycle $\xi \colon \Psi \to N$. Since $H^1(\Psi , N)$ is trivial, applying Lemma \ref{L:2} we see that $f$ is conjugate by an element of $N$
to the identity embedding $\Psi \to H$, $\delta \mapsto (1 , \delta)$, yielding our claim.
\end{proof}

\vskip1mm

\begin{lemma}\label{L:1}
Let $U$ be a unipotent group defined over a field $k$ of characteristic 0. Then for any finite group $\Gamma$ acting on $U(k)$, the cohomology set $H^1(\Gamma , U(k))$ is trivial.
\end{lemma}
\begin{proof}
If $U$ is commutative, then $U \simeq (\mathbb{G}_a)^d$ (cf. \cite[Ch. II, Remark 7.3]{Borel}). Then $U(k) \simeq k^d$ is a uniquely divisible abelian group, so $H^1(\Gamma , U(k)) = 0$.
In the general case, we argue by induction on $\dim U$. We may assume that $U$ is noncommutative, hence the center $V := Z(U)$ is a proper $k$-defined subgroup of positive dimension. We then have the exact sequence
$$
1 \to V \longrightarrow U \longrightarrow U/V \to 1
$$
of unipotent $k$-groups, where $V$ and $U/V$ are of dimension $< \dim U$. Since the Galois cohomology $H^1(\bar{k}/k , V)$ is trivial (cf., for example, \cite[Ch. 2, Lemma 2.7]{PlRa}), the sequence of $k$-points
$$
1 \to V(k) \longrightarrow U(k) \longrightarrow (U/V)(k) \to 1
$$
is also exact, i.e. $(U/V)(k)$ can be naturally identified with the quotient $U(k)/V(k)$. Moreover, the fact that  $U(k)$ is Zariski-dense in $U$ implies that $V(k)$ is precisely the center of $U(k)$, hence is invariant under the action of $\Gamma$. Then the action of $\Gamma$ on $U(k)$ also descends to $(U/V)(k) = U(k)/V(k)$. We have the following exact sequence of pointed sets
$$
H^1(\Gamma , V(k)) \longrightarrow H^1(\Gamma , U(k)) \longrightarrow H^1(\Gamma , (U/V)(k)).
$$
By the induction hypothesis, the sets $H^1(\Gamma , V(k))$ and $H^1(\Gamma , (U/V)(k))$ are trivial, so it follows from the above sequence that the set $H^1(\Gamma , U(k))$ is also trivial.
\end{proof}

\vskip1mm

\noindent {\bf Remark.} The assertion of Lemma \ref{L:1} remains valid if $p = \mathrm{char}\: k > 0$ if one assumes that $U$ is connected and $k$-split and the order of $\Gamma$ is prime to $p$. (Indeed, it follows from \cite[Proposition B.3.2]{CGP} that such a $U$ possesses a $k$-defined central subgroup isomorphic to $\mathbb{G}_a$.)

\vskip1mm

\begin{cor}\label{C:Conj}
Let $G$ be a semi-direct product  $N \rtimes H$, where $N$ is the group of $k$-points $U(k)$ of a unipotent group $U$ over a field $k$ of characteristic 0. Then every finite subgroup of $G$ is conjugate to a subgroup of $H$.
\end{cor}
\begin{proof} This follows from Lemmas \ref{L:3} and \ref{L:1}. \end{proof}

\section{Finite subgroups}

\noindent {\it Proof of Theorem \ref{T:I1}.} Let $\mathscr{B}$ be the product of buildings constructed for the reductive $k$-group $G$ over the field $K = k(\!(1/t)\!)$ in \S 2.4. We recall that the group $G(K)$ naturally acts on $\mathscr{B}$ by isometries, and therefore the resulting action of the group $\Gamma = G(k[t]) \subset G(K)$ on $\mathscr{B}$ is also isometric. According to Theorem \ref{T:FPT}, any finite subgroup $\Psi$ of $\Gamma$ fixes a point $x \in \mathscr{B}$. Applying Proposition \ref{P:genred}(1), we see that replacing $\Psi$ by a $\Gamma$-conjugate subgroup, we may assume that $x$ lies in the sector $\mathscr{Q}$ in the product $\mathscr{A}$ of standard apartments (cf. \S 2.4). According to Proposition \ref{P:genred}(2), the stabilizer $\Gamma_x$ is of the form $\mathscr{U}(k) \rtimes L(k)$ for some reductive $k$-subgroup $L$ of $G$ and some $k$-defined $k$-split unipotent group $\mathscr{U}$. Then by Corollary \ref{C:Conj}, the subgroup $\Psi$ is conjugate to a subgroup of $L(k) \subset G(k)$ within $\Gamma_x$, and the required fact follows. \hfill $\Box$

\vskip1mm

Before proving property (FC) for $G(k[t])$ over a $p$-adic field $k$, we first establish one finiteness result over a significantly broader class of fields. For this, we recall that according to Serre \cite[Ch. III, \S 4]{Serre-GC}, a profinite group $\mathscr{G}$ has {\it type} $(\mathrm{F})$ if it satisfies the following property:

\vskip2mm

\noindent $(\mathrm{F})$ {\it For every $n \geq 1$, the group $\mathscr{G}$ has only finitely many open subgroups of index $\leq n$.}

\vskip2mm

\noindent Furthermore, a field $k$ is of type $(\mathrm{F})$ if it is perfect and its absolute Galois group $\mathrm{Gal}(\bar{k}/k)$ is of type $(\mathrm{F})$. Examples of fields of type $(\mathrm{F})$ include $\mathbb{C}$, $\mathbb{R}$, and finite extensions of $\mathbb{Q}_p$ --- cf. {\it loc. cit.} For more examples of fields satisfying condition $(\mathrm{F})$ and its generalizations, we refer the reader to \cite{IR}.

\begin{prop}\label{P:F1}
Let $G$ be a connected linear algebraic group over a field $k$ of characteristic 0 that is of type $(\mathrm{F})$. Then for every finite group $\Gamma$, the group $G(k)$ has only finitely many conjugacy classes of subgroups isomorphic to $\Gamma$.
\end{prop}
\begin{proof}
Let $R = R(\Gamma , G)$ be the variety of representations of $\Gamma$ into $G$ (cf. \cite[2.4.7]{PlRa}). It is enough to show that the number of orbits of the adjoint action of $G(k)$ on $R(k)$ is finite. According to \cite[Theorem 2.17]{PlRa}, there are only finitely many orbits for the adjoint action of $G$ on $R$ and these orbits are Zariski-closed. Now, fix some $\rho \in R(k)$; then the corresponding orbit $V = G \cdot \rho$ is a $k$-defined homogeneous space of $G$.  It is well known that the orbits of $G(k)$ on $V(k)$ are in one-to-one correspondence with the elements of the kernel of the map $H^1(k , C) \to H^1(k , G)$, where $C$ is the centralizer of $\rho$ (cf. \cite[Ch. I, \S5.4, Corollary 1]{Serre-GC}). However, since $k$ is of type $(\mathrm{F})$, the set $H^1(k , C)$ is finite \cite[Ch. III, \S 4, Theorem 4]{Serre-GC}, and our claim follows.
\end{proof}

\vskip1mm

\begin{cor}\label{C:F1}
Let $G$ be an absolutely almost simple algebraic group defined over a field $k$ of characteristic 0 that is of type $(\mathrm{F})$. Then for every finite group $\Gamma$, the group $G(k[t])$ has only finitely many conjugacy classes of finite subgroups isomorphic to $\Gamma$.
\end{cor}
\begin{proof}
By Theorem \ref{T:I1} every finite subgroup of $G(k[t])$ is conjugate to a subgroup of $G(k)$. On the other hand, by Proposition \ref{P:F1}, the group $G(k)$ has only finitely many conjugacy subgroup isomorphic to $\Gamma$, so our assertion follows.
\end{proof}

\vskip1mm

\begin{prop}\label{P:F2}
Let $G$ be a reductive algebraic group over a finite extension $k$ of the $p$-adic field $\Q_p$. Then the group $G(k)$ has finitely many conjugacy classes of finite subgroups.
\end{prop}
\begin{proof}
As we already mentioned above, $k$ is of type $(\mathrm{F})$. So, it follows from Proposition \ref{P:F1} that it is enough to show that finite subgroups of $G(k)$ split into finitely many {\it isomorphism classes}. Considering a faithful $k$-defined representation $G \hookrightarrow \mathrm{GL}_n$, we see that it is enough to prove this fact for $G = \mathrm{GL}_n$. It is well known that every finite subgroup of $\mathrm{GL}_n(k)$ is conjugate to a subgroup of $\mathrm{GL}_n(\mathcal{O})$, where $\mathcal{O}$ is the valuation ring of
$k$ (cf. \cite[Proposition 1.12]{PlRa}). On the other hand, the congruence subgroup $\mathrm{GL}_n(\mathcal{O} , \mathfrak{p}^d)$, where $\mathfrak{p} \subset \mathcal{O}$ is the valuation ideal, is torsion-free for $d$ large enough (``Minkowski's Lemma," cf. \cite[p. 234]{PlRa}). Then every finite subgroup of $\mathrm{GL}_n(k)$ is isomorphic to a subgroup of the finite group $\mathrm{GL}_n(\mathcal{O}) / \mathrm{GL}_n(\mathcal{O} , \mathfrak{p}^d)$, and the required fact follows.
\end{proof}

\vskip1mm

\begin{thm}\label{T:FC-1}
Let $G$ be a reductive algebraic group defined over a finite extension $k$ of the $p$-adic field $\mathbb{Q}_p$. Then the group $G(k[t])$ has finitely many conjugacy classes of finite subgroups.
\end{thm}
\begin{proof}
This follows from Theorem \ref{T:I1} and Proposition \ref{P:F2}.
\end{proof}

\vskip2mm

\section{The Raghunathan--Ramanathan theorem}\label{S:RR}

\noindent {\bf \ref{S:RR}.1. The statement and preliminary remarks.} In \cite{RagRam}, Raghunathan and Ramanathan proved the following theorem.
\begin{thm}\label{T:RagRam}
Let $G$ be a connected reductive algebraic group over a field $k$, and let $\pi \colon B \to \mathbb{A}^1_k$
be a $G$-torsor over the affine line $\mathbb{A}^1_k = \mathrm{Spec}\: k[t]$. If $\pi$ is trivialized by the base change from $k$ to $\bar{k}$ (separable closure), then $\pi$ is obtained by the base change $\mathbb{A}^1_k \to \mathrm{Spec}\: k$ from a $G$-torsor $\pi_0 \colon B_0 \to \mathrm{Spec}\: k$.
\end{thm}

In this section, we present a simple proof of this theorem over fields of characteristic zero that is based on the Fixed Point Theorem; the argument in the general case requires some technical adjustments. First, we note that if $k$ has characteristic zero (or, more generally, is perfect) then every $G$-torsor  over the affine line is trivialized by the base change $\bar{k}/k$. In fact, this remains true if $\mathbb{A}^1_k$ is replaced by a Zariski-open subset, and actually by any geometrically connected smooth affine curve if $G$ is semi-simple.

\begin{lemma}\label{L:Triv}
Let $K$ be an algebraically closed field of characteristic 0, let $G$ be a connected reductive algebraic $K$-group, and let $C$ be a smooth affine curve over $K$. In each of the following situations

\vskip1mm

\noindent {\rm (1)} \parbox[t]{15cm}{$C$ is a Zariski-open subset of $\mathbb{A}^1_K$;}

\vskip1mm

\noindent {\rm (2)} \parbox[t]{15cm}{$G$ is semi-simple,}

\vskip1mm

\noindent we have $H^1(C , G) = 1$.
\end{lemma}
\begin{proof}
(Cf. the proof of Corollary 7.5 in \cite{CRR-GR}) Let $\beta \colon H^1(C , G) \to H^1(K(C) , G)$ be the map given by specialization to the
generic point. By Tsen's theorem, the field $K(C)$ has cohomological dimension $\leq 1$ (cf. \cite[Ch. II, 3.3]{Serre-GC}), so applying a theorem due to Steinberg \cite[Ch. III, 2.3]{Serre-GC}, we see that $H^1(K(C) , G) = 1$. So, it is enough to show that $\ker \beta$ is trivial. On the other hand, according to \cite{Nisn2}, there is a bijection between $\ker \beta$ and the double coset space
$$
\mathrm{cl}(G, K(C), V) := G(\mathbf{A}^{\infty}(V)) \backslash G(\mathbf{A}(V)) / G(K(C))
$$
where $G(\mathbf{A}(V))$ is the group of {\it rational} adeles of $G$ associated with the set $V$ of discrete valuations corresponding to the closed points of $C$, and $G(\mathbf{A}^{\infty}(V))$ and $G(K(C))$ are the subgroups of integral and principal adeles, respectively (see \cite{CRR-Isr} for a discussion of adeles in this context). Fix a maximal $K$-torus $T$ of $G$, which of course splits over $K$. A standard argument (cf. \cite[proof of Theorem 4.1]{CRR-Isr}) using strong approximation for the opposite maximal unipotent with respect to $V$ (which holds because $C$ is assumed to be affine) shows that every double coset in the class set $\mathrm{cl}(G, K(C), V)$ has a representative lying in $T(\mathbf{A}(V))$. We have a $K$-isomorphism $T \simeq \mathbb{G}_m^d$, where $d = \dim T$, and  the double coset space $\mathrm{cl}(\mathbb{G}_m, K(C), V)$ is a group that can be identified with the Picard group $\mathrm{Pic}\: C$.

Now, if $C$ is a Zariski-open subset of $\mathbb{A}^1_K$, then the group $P := \mathrm{Pic}\: C$ is trivial, and the triviality of $\ker \beta$ follows immediately. In the general case, we assume that $G$ is semi-simple, and let $\pi \colon \widetilde{G} \to G$ denote the universal cover. Since the curve $C$ is affine, there exists a canonical surjective group homomorphism $\mathbf{J}(K) \to P$, where $\mathbf{J}$ is the Jacobian of projective curve $\mathbf{C}$ that contains $C$.  It follows that the group $P$ is divisible, i.e. $nP = P$ for any $n \geq 1$. Consequently, for any $n \geq 1$, every coset in the class group $\mathrm{cl}(T, K(C), V)$ can be chosen in $T(\mathbf{A}(V))^n$. On the other hand, for $n = \vert \ker \pi \vert$, we have the inclusion
\begin{equation}\label{E:N}
G(\mathbf{A}(V))^n \subset \pi(\widetilde{G}(\mathbf{A}(V))).
\end{equation}
Again, since $\widetilde{G}$ is $K$-split, then arguing as in \cite[proof of Theorem 4.1]{CRR-Isr},  we conclude that $\tilde{G}$ has strong approximation and therefore $\widetilde{G}(\mathbf{A}(V)) = \widetilde{G}(\mathbf{A}^{\infty}(V))\widetilde{G}(K(C))$. It follows that $\pi(\tilde{G}(\mathbf{A}(V)) \subset G(\mathbf{A}^{\infty}(V)) G(K(C))$. Combining this with (\ref{E:N}), we see that
$$
G(\mathbf{A}(V))^n \subset G(\mathbf{A}^{\infty}(V)) G(K(C)),
$$
and, in particular, $T(\mathbf{A}(V))^n$ is contained in $G(\mathbf{A}^{\infty}(V)) G(K(C))$. Thus, we conclude that $\mathrm{cl}(G, K(C), V)$ reduces to a single element, and hence $\ker \beta$ is trivial.
\end{proof}

\vskip1mm

Turning now to the proof of Theorem \ref{T:RagRam} in the case where $\mathrm{char}\: k = 0$, we recall the Hochschild--Serre sequence
\begin{equation}\label{E:HS}
1 \to H^1(\mathrm{Gal}(\bar{k}/k) , G(\bar{k}[t])) \longrightarrow H^1(\mathbb{A}^1_k , G) \stackrel{\rho}{\longrightarrow} H^1(\mathbb{A}^1_{\bar{k}} , G),
\end{equation}
which is exact as a sequence of pointed sets (see \cite[2.2 and 2.9]{Gille} and also \cite[p. 96]{Milne-LEC} for the commutative case). The $G$-torsors  over $\mathbb{A}^1_k$ trivialized by the base change $\bar{k}/k$ are classified by the
the elements of $\ker \rho$. On the other hand, (\ref{E:HS}) yields a natural bijection
$$
\ker \rho \, \simeq \, H^1(\mathrm{Gal}(\bar{k}/k) , G(\bar{k}[t])).
$$
It follows that to complete the proof of Theorem \ref{T:RagRam}, it is enough to prove Theorem \ref{T:I2}. Furthermore, the map $H^1(\bar{k}/k , G(\bar{k})) \to H^1(\bar{k}/k , G(\bar{k}[t]))$ is injective by specialization argument (we can specialize at any $t \in k$, e.g. $t = 0$). So, we only need to prove that it is surjective.

\vskip1mm

\noindent {\bf \ref{S:RR}.2. Reduction to quasi-split simple adjoint groups.} We will now show that it is enough to prove Theorem \ref{T:I2} for quasi-split simple adjoint groups. First, we reduce to the case of adjoint groups. Given a reductive $k$-group $G$, there exists an exact sequence of algebraic groups
$$
1 \to F \longrightarrow G \stackrel{\pi}{\longrightarrow} \overline{G} \to 1
$$
in which $F$ is finite and $\overline{G}$ is the direct product $\overline{H} \times T$ of a semi-simple adjoint $k$-group $\overline{H}$ and a $k$-torus $T$. As in the proof of Corollary \ref{L:isog}, we see that Theorem \ref{T:Stavr1} implies the equality
$$
\overline{G}(\bar{k}[t]) = \overline{G}(\bar{k}) \cdot \pi_{\bar{k}[t]}(G(\bar{k}[t])),
$$
and since $\pi_{\bar{k}}(G(\bar{k})) = \overline{G}(\bar{k})$, we conclude that the sequence
\begin{equation}\label{E:ExSeq1}
1 \to F(\bar{k}) \longrightarrow G(\bar{k}[t]) \stackrel{\pi_{\bar{k}[t]}}\longrightarrow \overline{G}(\bar{k}[t]) \to 1
\end{equation}
is exact. We then have the following commutative diagram with exact rows
$$
\xymatrix{H^1(\bar{k}/k , G(\bar{k})) \ar[r]^{\gamma_1} \ar[d]_{\alpha} & H^1(\bar{k}/k , \overline{G}(\bar{k})) \ar[r] \ar[d]_{\beta} & H^2(\bar{k}/k , F(\bar{k})) \ar[d]_{=} \\ H^1(\bar{k}/k , G(\bar{k}[t])) \ar[r]^{\gamma_2} & H^1(\bar{k}/k , \overline{G}(\bar{k}[t])) \ar[r] & H^2(\bar{k}/k , F(\bar{k}))}
$$
If we assume that the map $H^1(\bar{k}/k , \overline{H}(\bar{k})) \to H^1(\bar{k}/k , \overline{H}(\bar{k}[t]))$ is surjective, then in view of the fact that $T(\bar{k}[t]) =
T(\bar{k})$ (which follows from the case of $\mathbb{G}_m$), the map $\beta$ will also be  surjective, hence bijective. So, by a diagram chase, we find that $\alpha$ is also surjective (one needs to use the fact that $H^1(\bar{k}/k , F(\bar{k}))$ acts transitively on the fibers of $\gamma_2$ --- cf. \cite[Ch. I, \S 5, Proposition 42, p. 54]{Serre-GC}).

Assume now that $G$ is semi-simple and adjoint. Writing $G$ as a direct product of $k$-simple groups, each of which is obtained by restriction of scalars from an absolutely simple group, we see that it is enough to consider the case where $G$ is an absolutely simple adjoint group. Such a group is an inner twist of a quasi-split group $G_0$ (cf. \cite[7.2]{Con} or \cite[Proposition 16.4.9]{Springer}), and we have compatible bijections
$$
H^1(\bar{k}/k , G(\bar{k}))  \simeq H^1(\bar{k}/k , G_0(\bar{k})) \ \ \text{and} \ \  H^1(\bar{k}/k , G(\bar{k}[t]))  \simeq H^1(\bar{k}/k , G_0(\bar{k}[t]))
$$
(cf. \cite[Ch. I, 5.3]{Serre-GC}). So, without loss of generality we may assume that $G$ is quasi-split.

\vskip1mm

\noindent {\bf \ref{S:RR}.3. Conclusion of the argument.} Thus, it remains to be shown that if $k$ is a field of characteristic 0, then {\it for an absolutely simple adjoint quasi-split $k$-group $G$ and a finite Galois extension $\ell/k$ with Galois group $\mathscr{G} = \mathrm{Gal}(\ell/k)$, the natural map
\begin{equation}\label{E:surj10}
H^1(\mathscr{G} , G(\ell)) \longrightarrow H^1(\mathscr{G} , G(\ell[t]))
\end{equation}
is surjective}.
For this we will use the natural action of the group $\Omega = \Gamma_{\ell} \rtimes \mathscr{G}$, where $\Gamma_{\ell} = G(\ell[t])$, on the Bruhat--Tits building $\mathscr{B}_{\ell}$ associated with $G$ over $L = \ell(\!(1/t)\!)$ --- see \S \ref{S:building} for the relevant notations.

Let $S$ be a maximal $k$-split torus of $G$, and $T$ be a maximal $k$-torus of $G$ containing $S$. Then the splitting field $\ell_0$ of $T$ is the minimal Galois extension of $k$ over which $G$ becomes split. It follows from the inflation-restriction sequence that without loss of generality, we may (and will) assume that $\ell \supset \ell_0$. 
Since $G$ is $k$-quasi-split, there exists a $k$-defined Borel subgroup $B$ of $G$ containing $T$. Then the system of simple roots $\Delta$ in the root system $\Phi(G , T)$ corresponding to $B$ is invariant under the Galois group $\mathscr{G}$.
Let $\mathscr{A}_{\ell}$ be the apartment in $\mathscr{B}_{\ell}$ corresponding to $T$, and let $\mathscr{Q}_{\ell}$ be the sector in $\mathscr{A}_{\ell}$ corresponding to the system of simple roots $\Delta$ (see \S \ref{S:building}). Since $T$ is $k$-defined, the apartment $\mathscr{A}_{\ell}$ is $\mathscr{G}$-invariant, and furthermore, since $\Delta$ is $\mathscr{G}$-invariant, $\mathscr{Q}_{\ell}$ is also $\mathscr{G}$-invariant.

Now, let $\xi \colon \mathscr{G} \to \Gamma_{\ell}$ be a 1-cocycle and let $$\iota_{\xi} \colon \mathscr{G} \to \Omega = \Gamma_{\ell} \rtimes \mathscr{G},  \ \ \sigma \mapsto (\xi(\sigma) , \sigma)$$ be the corresponding embedding (see Lemma \ref{L:2}). Since the natural action of $\mathscr{G}$ on $\mathscr{B}_{\ell}$ is isometric (due to the fact that the action of $\mathscr{G}$ on $L$ preserves the valuation), so is the action of $\Omega$, hence  using Theorem \ref{T:FPT}, we can find a point  $x \in \mathscr{B}_{\ell}$ fixed by the finite subgroup $\iota_{\xi}(\mathscr{G}) \subset \Omega$. As we have seen in \S \ref{S:building}.3, Theorem \ref{T:Marg} remains valid for not necessarily simply connected groups, so there exists $g \in \Gamma_{\ell}$ such that $y = gx \in \mathscr{Q}_{\ell}$. Thus, replacing $\iota_{\xi}$ with the conjugate embedding $g \iota_{\xi} g^{-1}$, which is of the form $\iota_{\xi'}$ for a cocycle $\xi'$ that is equivalent to $\xi$ (see the proof of Lemma \ref{L:2}), we may (and will) assume that $\iota_{\xi}(\mathscr{G})$ fixes a point $y \in \mathscr{Q}_{\ell}$. Let $F = F_y$ be the carrier of $y$ (closed simplex containing $y$ in its interior). Clearly, $F \subset \mathscr{Q}_{\ell}$ and $F$ is invariant under $\iota_{\xi}(\mathscr{G})$. This means that for any $\sigma \in \mathscr{G}$, we have
\begin{equation}\label{E:last}
\xi(\sigma)(\sigma(F)) = F.
\end{equation}
Since $\mathscr{Q}_{\ell}$ is $\mathscr{G}$-invariant, we conclude that $\sigma(F) \subset \mathscr{Q}_{\ell}$, and then, in view of the uniqueness
in Theorem~\ref{T:Marg}, the condition (\ref{E:last}) implies  that $\sigma(F) = F$. Thus, $F$ is $\mathscr{G}$-invariant. But then $\xi(\sigma)(F) = F$, so since the action of $\Gamma$ is type-preserving (cf. Proposition \ref{P:TP}), we obtain from Lemma \ref{L:TP}
that $\xi(\sigma)$ fixes $F$ pointwise. Thus, $\xi$ is a cocycle with values in $H(F)$, the pointwise stabilizer of $F$.

According to Proposition \ref{P:stab2}, the stabilizer $H(F)$ is of the form $\mathscr{U}(\ell) \rtimes L(\ell)$ where $L \subset G$ is a $k$-defined reductive subgroup and $\mathscr{U}$ is a $k$-defined $k$-split unipotent group. Since the map $H^1(\mathscr{G} , L(\ell)) \to H^1(\mathscr{G} , \mathscr{U}(\ell) \rtimes L(\ell))$ is bijective (cf. \cite[Ch. II, Proposition 2.9]{PlRa}), we see that $\xi$ is equivalent to a cocycle with values values in $L(\ell)$, hence lies in the image of the map $H^1(\ell/k , G(\ell)) \to H^1(\ell/k , G(\ell[t]))$, as required.

\vskip5mm

\vskip3mm

\vskip2mm

\noindent {\small {\bf Acknowledgements.} We are grateful to Vladimir Chernousov and Gopal Prasad for useful comments and suggestions. Thanks are also due to the anonymous referee who carefully read the paper and proposed a number of improvements. The third-named author was partially supported by NSF grant DMS-2154408.}

\vskip5mm

\bibliographystyle{amsplain}

\end{document}